\input ./style/arxiv-general.cfg
\documentclass[aap,MSNbibl,seceqn,dvips]{arximspdf}
\makeatletter
   \@ifpackageloaded{graphicx}{}{\usepackage{graphicx}}
\makeatother
\usepackage{dcolumn}
\usepackage{mathrsfs}

\doi{10.1214/15-AAP1143}
\volume{26}
\issue{4}
\pubyear{2016}
\firstpage{2169}
\lastpage{2192}
\docsubty{FLA}

\makeatletter
\renewcommand{\top}{{T}}
\newcolumntype{d}[1]{D{.}{.}{#1}}

\newcommand{\rrvert}{\vert}
\newcommand{\llvert}{\vert}
\newcommand{\rpspace}{\mathscr{C}}
\newcommand{\implies}{\Longrightarrow}
\newcommand{\eqref}[1]{(\ref{#1})}
\newtheorem{theorem}{Theorem}
\newproclaim{definition}[theorem]{Definition}
\newproclaim{example}[theorem]{Example}
\newtheorem{proposition}[theorem]{Proposition}
\newproclaim{remark}[theorem]{Remark}
\newtheorem{lemma}[theorem]{Lemma}
\newcommand{\FullMeasureSet}{\mathbb{D}}
\newcommand{\R}{\mathbb{R}}
\renewcommand{\P}{\mathbb{P}}
\makeatother

\begin{document}
\begin{frontmatter}

\title{Pathwise stability of likelihood estimators for diffusions via
rough paths}
\runtitle{Pathwise stability of likelihood estimators}

\begin{aug}
\author[A]{\fnms{Joscha}~\snm{Diehl}\thanksref{M1,M2,T1}\ead[label=e0]{joscha.diehl@gmail.com}},
\author[B]{\fnms{Peter}~\snm{Friz}\corref{}\thanksref{M1,M3,T2}\ead[label=e1]{friz@math.tu-berlin.de}\ead[label=e11]{friz@wias-berlin.de}}
\and
\author[C]{\fnms{Hilmar}~\snm{Mai}\thanksref{M4,T2}\ead[label=e2]{hilmar.mai@ensae.fr}}
\runauthor{J. Diehl, P. Friz and H. Mai}
\affiliation{TU Berlin\thanksmark{M1},
University of California, San Diego\thanksmark{M2},
WIAS Berlin\thanksmark{M3}\\ and
ENSAE ParisTech\thanksmark{M4}}
\address[A]{J. Diehl\\
TU Berlin\\
Strasse des 17. Juni 136\\
10623 Berlin\\
Germany\\
and\\
Department of Mathematics---MC 0112\\
University of California, San Diego\\
9500 Gilman Drive\\
La Jolla, California 92093\\
USA\\
\printead{e0}}
\address[B]{P. Friz\\
TU Berlin\\
Strasse des 17. Juni 136\\
10623 Berlin\\
Germany\\
and\\
Weierstrass Institute for\\
\quad Applied Analysis and Stochastics\\
Mohrenstrasse 39\\
10117 Berlin\\
Germany\\
\printead{e1}\\
\phantom{E-mail: }\printead*{e11}}
\address[C]{H. Mai\\
CREST ENSAE ParisTech\\
3, av. Pierre Larousse\\
92240 Malakoff\\
France\\
\printead{e2}}
\end{aug}
\thankstext{T1}{Supported by DFG, SPP 1324.}
\thankstext{T2}{Supported in part by European Research Council under
the European Union's Seventh Framework
Programme (FP7/2007-2013)/ERC Grant agreement nr. 258237.}

%
\received{\smonth{3} \syear{2015}}

%
\begin{abstract}
We consider the classical estimation problem of an unknown drift
parameter within classes of nondegenerate diffusion processes.
Using rough path theory (in the sense of T. Lyons), we analyze
the Maximum Likelihood Estimator (MLE) with regard to its pathwise
stability properties as well as robustness
toward misspecification in volatility and even the very nature of the noise.
Two numerical examples demonstrate the practical relevance of our results.
\end{abstract}

%
\end{frontmatter}

\setcounter{footnote}{2}

\section{Introduction}

Let $W$ be $d$-dimensional Wiener process and $A\in\mathbb{V}$, some
fixed finite-dimensional vector space.\vspace*{1pt}
Consider sufficiently regular $h:
\mathbb{R}^{d}\rightarrow L(\mathbb{V},\mathbb{R}^{d})$ and $\Sigma
:\mathbb{R}%
^{d}\rightarrow L ( \mathbb{R}^{d},\mathbb{R}^{d} ) $ so
that\footnote{$L(V,W)$ denote the space of linear maps between two
finite-dimensional vector spaces and may be identified with a space of
matrices.}
%
\begin{equation}
\label{sde} dX_t = h ( X_t ) A \,dt+ \Sigma (
X_t ) \,dW_{t}
\end{equation}
has a unique solution, started from $X_{0}=x_{0}$. Throughout we shall assume
nondegeneracy of the matrix-valued diffusion coefficient $\Sigma$.
The important example of
multidimensional Ornstein--Uhlenbeck dynamics, for instance, falls in the
class of diffusions considered here.
We are interested in estimating the drift parameter~$A$, given some
observation sample path $ \{ X_{t} (
\omega ) =\omega_{t}:t\in [ 0,T ]  \} $. To
this end, we consider the classical Maximum Likelihood Estimator (MLE),
of the form
\[
\hat{A}_T = \hat{A}_T ( \omega ) = \hat{A}_T
\bigl(X(\omega )\bigr) \in\mathbb{V}
\]
relative to the reference measure given by the law of the drift-free
process. Note that $\hat{A}_T$ is
a functional on pathspace $C([0,T],\mathbb{R})$: for every
(observation) sample path $X(\omega) = \omega$
one has a corresponding estimate $\hat{A}_T(X(\omega))$. Let us also
recall that these MLEs are based
on the Girsanov density of the pathspace measures, with- versus
without-drift, respectively, see, for example, the standard
text books of Kutoyants \cite{Kutoyants2004} or Liptser--Shiryaev
\cite{Liptser2001}.
It will be instructive to consider the simplest possible example with
its fully explicit solution.
%
\begin{example}[(Scalar Ornstein--Uhlenbeck process)]\label{Ex1} In (\ref{sde}), take $d=1$, $h(x)
= x$,
$\Sigma\equiv\sigma> 0$ and scalar drift parameter $A$. The MLE for
this parameter is then given explicitly by
%
\begin{equation}
\hat A_T (X) = \frac{X_T^2 - x_0^2 - \sigma^2 T}{2 \int_0^T X_t^2   \,dt}.
\end{equation}
\end{example}
Despite its simplicity, the above example exhibits a few interesting
properties: First, it is not well defined for every possible path $X
\in C([0,T],\mathbb{R})$, and indeed $X \equiv0$ leaves us with an
ill-defined division by zero. Second, provided we stay away from the
zero-path, we have pathwise stability in the sense that two observation
$X$ and $\tilde{X}$ which are uniformly close on $[0,T]$ give rise to
close estimations $\hat{A}_T (X) \approx\hat{A}_T(\tilde{X})$. (In
other words, the functional $\hat{A}_T$ is continuous on
$C([0,T],\mathbb{R})- \{ 0 \}$, with respect to the uniform topology.)
At last, the estimator depends continuously on the parameter $\sigma$,
despite the fact that pathspace measures associated to different values
of $\sigma$ are actually mutually singular.\footnote{The laws of
$\sigma W$ and $\tilde\sigma W$ are mutually singular when $\sigma\ne
\tilde\sigma$; just compute the quadratic variation.}

In order to understand such stability considerations in greater
generality, we now review the MLE construction for a general diffusion
as given in \eqref{sde}.
To this end, recall that by Girsanov's theorem, under the
standing assumption that $C:=\Sigma\Sigma^{T}$ is nondegenerate, the
corresponding measures on pathspace, say $\mathbb{P}^{A,\Sigma}$ and
$\mathbb{P}^{0} := \mathbb{P}^{0,\Sigma}$, are absolutely
continuous so that the MLE method is applicable. Standard computations,
partially reviewed below, show that one has
\[
I_{T}\hat{A}_{T}=S_{T},
\]
where
\[
S_{T}=\int_{0}^{T}h(X_s)^{\top}C^{-1}(X_{s})
\,dX_{s} \in\mathbb{V} 
^{\ast}
\]
and
\[
I_{T}=\int_{0}^{T}h(X_s)^{\top}C^{-1}(X_{s})h(X_s)
\,ds\in L \bigl(\mathbb{V},\mathbb{V}^{\ast} \bigr),
\]
where the first integral above (against $dX$) is understood in It\^o
sense. Of course, degeneracy may be a problem, for instance, when
$h\equiv0$.
One the other hand, for ``reasonable'' nondegenerate $h$ [such as
$h (
x ) =x$ in the Ornstein--Uhlenbeck model case] one can expect
a.s. invertibility of $%
I_{T}$ and thus an a.s. well-defined estimator
%
\begin{equation}
\hat{A}_{T} ( \omega ) =I_{T}^{-1}S_{T}
\in\mathbb{V}. \label{MLEclassical}
\end{equation}
Let us emphasize that $S_{T}$ involves a
stochastic (here: It\^{o}) integral so that $S_{T}$ is also only
defined up
to null-sets. At this stage, one has (at best) a measurable map $\hat
{A}%
_{T}:C (  [ 0,T ] ,\mathbb{R}^{d} ) \rightarrow
\mathbb{V}$
with the usual null-set ambiguity.\footnote{%
The situation is reminiscent of SDE theory: the It\^{o}-map is also a
measurable map on pathspace, in general only defined up to null-sets.}
The following questions then arise rather naturally:
\begin{longlist}[(Q2)]
\item[(Q1)] Under what conditions on $h$ (and $\omega$) is
$I_{T}=I_{T} (
X ( \omega )  ) $ actually invertible? [For $\mathbb
{P}^{0}$-a.e.
$X ( \omega )=\omega$, say, or provide a robust pathwise
condition.]

\item[(Q2)] Assuming suitably invertibility of $I_{T}$, so that the estimator
$\hat{A}_{T}$ is well defined, do we have ``robustness'' of\vspace*{1pt} the estimate
problem in the following sense: if $X \approx\tilde{X}$ (e.g., in the
sense that $\sup_{t\in [ 0,T ] }\llvert  X_{t}-\tilde{%
X}_{t}\rrvert \ll1$ or perhaps a more complicated metric) is it true
that
\[
\hat{A}_{T} ( X ) \approx\hat{A}_{T}( \tilde{X}%
)
?
\]
In other words, is the functional $\hat{A}_{T}$ continuous in some topology?

\item[(Q3)] Write $\hat{A}_{T}^{\sigma}$ to indicate the MLE under volatility
specification $\Sigma=\sigma I$. Assume we are not entirely certain about
the value of $\sigma$. Is it true---a rather sensible request from a user's
perspective---that
\[
\sigma\approx\tilde{\sigma}\implies\hat{A}_{T}^{\sigma}\approx
\hat{A}%
_{T}^{\tilde{\sigma}} ?
\]
\end{longlist}

We emphasize that (Q3) is a difficult question, last not least because the
respective pathspace measures are singular whenever $\sigma\neq\tilde
{%
\sigma}$. Hence, it is not even clear if one is allowed to speak
``simultaneously'' of $\hat{A}_{T}^{\sigma}$ for all $\sigma
$.\footnote{%
The situation is reminiscent of stochastic flow theory: for each fixed
starting point, SDE solution may be (well-) defined (up to null-sets), but
it is far from clear---and not true in general in infinite dimension---that
one can define solutions for all starting points on a common set of
full measure. The financial theory of \textit{uncertain volatility} (see
\cite{Avellaneda1995} and \cite{Lyons1995}) also poses related
problems.} The situation becomes even worse if one considers all possible
volatility specifications in a class like
\[
\bigl\{ \Sigma: c^{-1}I\leq\Sigma\Sigma^{T}\leq cI \bigr\}
.
\]
Indeed, this space is infinite-dimensional, leaving no hope to ``fix'' things
with Kolmogorov-type criteria. On the other hand, explicit computations
(e.g., in the Ornstein--Uhlenbeck case, Example \ref{Ex1} and
Section~\ref{sec:counterexamples}) show that $\hat{A}$ is
extremely well behaved in $\sigma$. Hence, we can certainly hope for
some sort of robustness of the MLE with respect to the volatility
specification.

The last question we would like to investigate is about
misspecification of the noise $W$. The assumption of independent
increments of $W$ is a strong limitation in applications and a
nontrivial dependence structure in time appears in many real data examples.
\begin{longlist}[(Q4)]
\item[(Q4)] Suppose that the model is misspecified in the sense that
\eqref{sde} is in fact driven by a fractional Brownian motion $W^H$
with Hurst index $H$. Is the MLE $\hat A_T$ robust in some sense (e.g.,
when $H \approx1/2$) with respect to this change of the model?
\end{longlist}

Our main theorem in Section~\ref{sec:main} provides reasonable answers
to question (Q1) to (Q3) based on T. Lyons' rough path theory \cite
{LQ02,LCT,FV,FH}, a short review of which will be given in Section~\ref
{rpreview} below. It is worth emphasizing that the rough path ideas are
pivotal to the pathwise robustness results obtained here: in
Section~\ref{sec:counterexamples}, we give an explicit example
illustrating the failure of robustness if one uses the usual uniform
topology. Question (Q4) will be addressed in Section~\ref
{sec:misspecification}.

In Section~\ref{sec:numerical}, we present two numerical examples
demonstrating the practical value of our theoretical results.
The first one concerns an Ornstein--Uhlenbeck process driven by
physical Brownian motion in a magnetic field.
As was recently demonstrated in \cite{FGL}, physical Brownian motion
in a magnetic field does \emph{not} converge---in the zero mass limit---to standard Brownian motion
on the level of rough paths; a correction term appears.
Nonetheless, our main theorem tells us how to appropriately correct the
estimator for the
OU process driven by a standard Brownian motion in order to still get
reasonable results.
The second example concerns the Ornstein--Uhlenbeck process driven by
fractional Brownian motion $W^H$
with Hurst parameter $H$.
For $H < 1/2$ naively applying the classical estimator is not
well-posed, since the It\^o integrals
are not well defined.
There exists, though, a canonical rough path lift for $H > 1/3$ and
plugging this
into the estimator of Theorem \ref{thm:mainTheorem} leads to
surprisingly good results even for $H\neq1/2$.
The theoretical background for this example is presented in
Section~\ref{sec:misspecification};
most importantly, the fact that the rough path lift is continuous in $H$.

The interplay of statistics and rough paths is very recent. The first
and (to our knowledge) only paper is \cite{PL} where the authors
consider general rough differential equations driven by random rough
paths and propose parametric estimation of the coefficients
based on Lyons' notion of expected signature. That said, the present
paper constitutes the first use of rough path analysis toward
robustness questions related
to classical statistical estimation problems for diffusion processes.

\section{A first step: Stratonovich estimators}
\label{sec:stratonovich}
Let us recall a few basic facts about convergence of discrete
approximations of stochastic integrals.
This is a central issue when applying the maximum likelihood estimators
in the context of discrete observations and will be of importance for
our numerical examples in Section~\ref{sec:numerical}.

Let $X$ be a (possibly multi-dimensional) continuous semimartingale.
Then for regular enough functions $f$:
\begin{longlist}[(iii)]
\item[(i)]
the left-point Riemann sums converge to the It\^o integral in probability
\begin{eqnarray*}
&&\sum_{[u,v] \in\mathcal{P}^n } f(X_u) [ X_v -
X_u ] \to_{n\to\infty} \int f(X_r) \,dX_r,
\end{eqnarray*}
for any sequence $\mathcal{P}^n$ of partitions with mesh-size going to $0$;

\item[(ii)]
the trapezoidal Riemann sums converge to the Stratonovich integral in
probability
\begin{eqnarray*}
\sum_{[u,v] \in\mathcal{P}^n } \frac{1}{2} \bigl[
f(X_u) + f(X_v) \bigr] [ X_v -
X_u ] \to_{n\to\infty} \int f(X_r) \circ
dX_r,
\end{eqnarray*}
for any sequence $\mathcal{P}^n$ of partitions with mesh-size going to $0$;

\item[(iii)]
for any reasonable\footnote{For example, piecewise linear, mollified, etc.}
smooth approximations $X^n \to_{n \to\infty} X$ the corresponding
classical Riemann--Stieltjes integrals converge to the Stratonovich
integral in probability
\begin{eqnarray*}
\int f\bigl(X^n_s\bigr) \,dX^n_s
= \int f\bigl(X^n_s\bigr) \dot X^n_s
\,ds \to_{n\to\infty} \int f(X_r) \circ \,dX_r.
\end{eqnarray*}
\end{longlist}
The first point illustrates how a MLE is usually used in practice, for
discrete time observations:
Since the process $X$ is only known at a finite number of time points
(discrete observations), the stochastic integrals are usually
approximated by left-point Riemann sums.
This is in fact a quite unstable procedure, as will be illustrated in
Section~\ref{sec:misspecification}.

On the other hand, looking at (iii), it is reasonable to
expect, that any positive answer to (Q2) will
start out with the Stratonovich formulation of the MLE:
%
\begin{equation}
\label{eq:forMLE} I_{T}(X) \hat{A}_{T}(X) =S^{\mathrm{Strat}}_{T}(X),
\end{equation}
where
\begin{eqnarray*}
S^{\mathrm{Strat}}_{T}(X)_{i,j} &=& 
\int
_{0}^{T} h_i(X_{s})^T
C^{-1}_{j \cdot}(X_{s}) \circ \,dX_{s}
\\
&&{}-
\frac{1}{2} \int_0^T \operatorname{Tr}
\bigl[ D\bigl( h_i C^{-1}_{j \cdot
}\bigr)
(X_s) \Sigma(X_s) \Sigma(X_s)^T
\bigr] \,ds,
\\
I_{T}(X)&=&\int_{0}^{T}h(X_s)^{\top}C^{-1}(X_{s})h(X_s)
\,ds\in L \bigl( \mathbb{V},\mathbb{V}^{\ast} \bigr).
\end{eqnarray*}
There is, at first, only a notational difference between $S^{\mathrm{Strat}}$
and $S$,
since we have just rewritten the It\^o integral as a Stratonovich one.
Taking a hint from point (iii) above though, we \emph
{define} from now on $S^{\mathrm{Strat}}_T(X)$ for smooth paths $X$
(a null-set under the diffusion measure) with the Stratonovich
integrals replaced by Riemann--Stieltjes integrals.
Before stating our first stability result which justifies this
definition, we give the following well-posedness result on the estimator.
We assume this result to be folklore in the statistical community, but
were unable to find a relevant reference.
The proof will be given in Section~\ref{sec:proofs}.
%
\begin{proposition}
\label{prop:wellposednessMLE}
The MLE for $A$ in equation \eqref{sde} is characterized by \eqref{eq:forMLE}.
Moreover, if we define
\[
R_{h}:= \bigl\{ X\in C \bigl( [ 0,T ] ,\mathbb {R}^{d}
\bigr) : \forall M \in\mathbb{V}, M \neq0, \exists t \in[0,T]\mbox{ s.t. } h (
X_{t} ) M \neq0 \bigr\},
\]
and assume that
\[
\mathbb{P}^{0,\Sigma} ( R_{h} ) =1,
\]
then $I_{T}=I_{T} ( X  ) $ is $\mathbb{P}^{0,\Sigma}$-almost surely invertible so that $A_{T}= A_{T} ( X  ) :=
I_{T}^{-1}S_{T}(X)$ is $\mathbb{P}^{0,\Sigma}$-almost surely
well defined.
\end{proposition}

We then have the following first stability result. 
%
\begin{proposition}
\label{prop:stratonovichRobust}
Assume that $\P^{0,\Sigma}(R_h) = 1$,
and let $X^n$ be piecewise linear approximations to $X$ such that
$R^h$ has full measure under the image measure of $X^n$ for all $n$.
Then in probability
\begin{eqnarray*}
\hat A^{\mathrm{Strat}}_T\bigl( X^n \bigr) :=
I_{T}^{-1}\bigl(X^n\bigr) S^{\mathrm{Strat}}_{T}
\bigl(X^n\bigr) \to_{n\to
\infty} \hat A_T(X).
\end{eqnarray*}
\end{proposition}

\begin{pf}
The claimed stability (in probability) of Stratonovich integrals, which
can be found, for example, in Section~6.6 in \cite{bibikedaWatanabe}
(see \cite{FH}, Section~9.2, for a modern proof), yields
$S^{\mathrm{Strat}}_T(X^n) \to S^{\mathrm{Strat}}_T(X)$ as $n\to\infty$.
Moreover $I_T^{-1}$ is continuous in supremum norm in $X$ on $R_h$, and
hence the statement follows.
\end{pf}

Note that the preceding result only concerns convergence in probability;
it therefore does \emph{not} provide a good answer to (Q2).
To wit, for the Stratonovich estimator $\hat A_T^{\mathrm{Strat}}$ it is in general
\emph{not} true that paths that are uniformly close in supremum norm,
that the resulting estimates will be close. We give an explicit
(deterministic) counterexample in Section~\ref{sec:counterexamples}.
A stochastic counterexample will be given in Section~\ref{sec:numerical},
in the setting of a physical Brownian motion in a magnetic field.

In order to fix this problem (and in order to answer the other questions),
we will adopt a rough path perspective in the next section.
To this end, we now give some recalls on rough path theory.

\section{Brief review of rough paths}\label{rpreview}

In this section, we introduce some basic notions from Lyons' rough
paths theory. Our notation here follows Friz--Hairer \cite{FH}, which
is also a
source of much more on this material, together with the standard
references \cite{LQ02,LCT,FV}.

We start by giving a definition of H\"older continuous rough paths that
is suitable for our purpose. Let $X: [0,T] \to\mathbb{R}^d$ be a
\emph{smooth} path and define the second-order iterated integrals
$\mathbb{X}: [0,T]^2 \to\mathbb{R}^d \otimes\mathbb{R}^d$ of $X$ by
\[
\mathbb{X} _{s,t}:=\int_{s}^{t}X_{s,r}
\otimes dX_{r},
\]
where $X_{s,r}=X_{r}-X_{s}$ are the increments of $X$.
Then the pair $(X,\mathbb{X})$ has the analytic property
\[
( \mathrm{ANA} ) _{\alpha}:\cases{ \llvert X_{s,t}\rrvert \lesssim\llvert
t-s\rrvert ^{\alpha},
\vspace*{3pt}\cr
\llvert \mathbb{X}_{s,t}\rrvert \lesssim
\llvert t-s\rrvert ^{2\alpha}}
\]
for any $\alpha\leq1$ and satisfies the algebraic relation
\begin{eqnarray*}
(\mathrm{ALG}) :\mathbb{X}_{s,t}+X_{s,t}\otimes X_{t,u}+
\mathbb {X}_{t,u}&=&\mathbb{X}_{s,u},
\\
\bigl(\mathrm{ALG}^{\prime}\bigr) :2\operatorname{Sym} ( \mathbb{X}_{s,t}
) &=&X_{s,t}\otimes X_{s,t},
\end{eqnarray*}
for $s,t,u \in[0,T]$.
More generally speaking, these two conditions are used to define a
\textit{rough path} in $\mathbb{R}^d$.
%
\begin{definition}
Fix $\alpha\in(1/3,1/2]$. Any $\mathbf{X}= ( X,\mathbb
{X} ) $ for
which $ ( \mathrm{ANA} ) _{\alpha}+(\mathrm{ALG})$ holds is called (weak
$\alpha$-H\"{o}lder) {rough path}. If also $ ( \mathrm{ALG}^{\prime} ) $
is satisfied call it {geometric}. The space of $\alpha$-H\"older rough
paths and its subset of geometric rough paths are denoted by $%
\rpspace^{\alpha} (  [ 0,T ] ,\mathbb{R}^{d} )
$ and $\rpspace_{g}^{\alpha} (  [ 0,T ] ,\mathbb
{R}^{d} )$, respectively.
\end{definition}
Rough paths arise naturally as sample paths of stochastic processes.
The basic example is a $d$-dimensional Brownian motion $B$ enhanced
with its iterated integrals
\[
\mathbb{B}_{s,t} := \int_s^t
B_{s,r} \otimes dB_r \in\mathbb{R}^{d
\times d},
\]
where the integral on the right-hand side can be understood in It\^o or
Stratonovich sense leading to It\^o or Stratonovich enhanced Brownian
motion, respectively. Then with probability one $\mathbf{B} =
(B,\mathbb{B}) \in\rpspace^{\alpha} (  [ 0,T ]
,\mathbb{R}^{d} )$ for any $\alpha\in(1/3,1/2)$ and \mbox{$T >0$}. We
also say that we can lift $B$ to a rough path $\mathbf{B}$ by adding
the second-order terms $\mathbb{B}$. A similar rough paths lift is
given in our main result for the solution of \eqref{sde}.

To investigate stability questions for the parameter estimation problem
in a pathwise sense, we need suitable metric on $\rpspace^{\alpha
} (  [ 0,T ] ,\mathbb{R}^{d} ) $. It turns out
that an adequate metric on $\rpspace^{\alpha} (  [
0,T ] ,\mathbb{R}^{d} )$ can be defined from $ (
\mathrm{ANA} ) _{\alpha}$ as follows.
%
\begin{definition}
\label{def:holderSpace}
For $\mathbf{X,Y} \in\rpspace^{\alpha} (  [ 0,T ]
,\mathbb{R}^{d} )$ the $\alpha$-H\"older rough path metric is
given by
\[
\rho_\alpha(\mathbf{X,Y}) := \sup_{s \neq t \in[0,T]}
\frac
{\llvert X_{s,t} - Y_{s,t}\rrvert }{\llvert t-s\rrvert ^\alpha} + \sup_{s \neq t \in[0,T]} \frac
{\llvert \mathbb{X}_{s,t} - \mathbb{Y}_{s,t}\rrvert }{\llvert t-s\rrvert ^{2\alpha}}.
\]
\end{definition}
%
\begin{remark}
In the original formulation of rough paths theory in \cite{L98},
distance was measured in $p$-variation norm instead of the $\alpha$-H\"
older norm used here.
The results in this work can be rephrased without difficulty in a
$p$-variation setting.
This applies in particular to the continuity of the map $\hat{\mathbf{A}}_{T}$ in Theorem \ref{thm:mainTheorem}(ii) and (iii) below.
\end{remark}
We conclude this section with rough integrals and its relation to
stochastic integration. Let $\mathcal{P}$ be a partition of $[0,T]$
and denote by $|\mathcal{P}|$ the length of its largest element. For
$\mathbf{X} = (X,\mathbb{X}) \in\rpspace^{\alpha} (  [
0,T ] ,\mathbb{R}^{d} )$ and $\alpha> 1/3$ we aim at
integrating $F(X)$ for $F \in\mathcal{C}_b^2 (\mathbb{R}^d,\mathcal
{L}(\mathbb{R}^d,\mathbb{R}^m))$ against $\mathbf{X}$. It is well
known that classical Young integration is possible for expressions of
the form
\[
\int_0^T F(X_t) \,dX_t
\]
only if $X \in\mathcal{C}^\alpha$ for $\alpha> 1/2$. This excludes,
for example, paths of Brownian motion which are of order $\alpha< 1/2$.
This barrier was overcome by rough paths theory by taking into account
``second order'' terms. Indeed, one can show that the limit in
\[
\int_0^T F(X_s) \,d
\mathbf{X}_s := \lim_{|\mathcal{P}| \to0} \sum
_{(s,t) \in\mathcal{P}} F(X_s) X_{s,t} + D
F(X_s) \mathbb{X}_{s,t}
\]
exists and is called a (Lyons') rough integral \cite{L98}.
Most importantly for us, rough integrals depend continuously in rough
path metric on $\mathbf{X}$
and
by taking $\mathbf{X} = \mathbf{B}$ to be (Stratonovich) enhanced
Brownian motion one recovers with probability one the Stratonovich integral.

We shall need the following standard result; see, for example,
Friz--Victoir \cite{FV}, Section~13, or Friz--Hairer \cite{FH}, Section~10.
%
\begin{proposition}\label{prop:lift}
Fix $\alpha\in ( 1/3,1/2 ) $. Then, $\mathbb{P}^{0,\Sigma
}$-almost
surely, $X ( \omega ) $ lifts to a (random) geometric
$\alpha$-H%
\"{o}lder rough path, that is, a random element in the rough path space
$%
\rpspace_{g}^{\alpha} (  [ 0,T ] ,\mathbb
{R}^{d} )$ (as reviewed in the next section),
via the (existing) limit in probability
\[
\mathbf{X} ( \omega ) := \bigl( X ( \omega ) ,\mathbb{X}%
 ( \omega )
\bigr) :=\lim_{n} \biggl( X^{n},\int
X^{n}\otimes dX^{n} \biggr),
\]
where $X^{n}$ denotes dyadic piecewise linear approximations to $X$.
\end{proposition}

\section{Main result} \label{sec:main}

We are now ready to formulate our main result.
By constructing an estimator on rough path space, we resolve the
pathwise stability problem
that is inherent to the Stratonovich estimator (compare Proposition
\ref{prop:stratonovichRobust} and Section~\ref{sec:counterexamples}).
%
\begin{theorem}
\label{thm:mainTheorem}
Assume that $\mathbb{P}^{0,\Sigma} ( R_{h} ) = 1$, so that
the MLE $\hat A$ is well defined by Proposition \ref{prop:wellposednessMLE}.

\begin{longlist}[(iii)]
\item[(i)]
Define $\mathbb{D} \subset\rpspace_{g}^{\alpha} (  [ 0,T%
 ] ,\mathbb{R}^{d} ) $ by
%
\[
\FullMeasureSet= \bigl\{ ( X,\mathbb{X} ) \in\rpspace _{g}^{\alpha
}:X
\in R_{h} \bigr\}.
\]
Then for every fixed, nondegenerate volatility
function $\Sigma$,
\[
\mathbb{P}^{0,\Sigma} \bigl[ \mathbf{X} ( \omega ) \in \FullMeasureSet
\bigr] =1.
\]

\item[(ii)]
There exists a deterministic, continuous (with respect to $\alpha$-H%
\"{o}lder rough path metric; see Definition \ref{def:holderSpace}) map
\[
\hat{\mathbf{A}}_{T}:\cases{ \FullMeasureSet \rightarrow \mathbb{V},
\vspace*{3pt}\cr
\mathbf{X} \mapsto \hat{\mathbf{A}}_{T} ( \mathbf{X} )}
\]
so that, for every fixed, nondegenerate volatility function $\Sigma$,
%
\begin{equation}
\label{eq:identity} \mathbb{P}^{0,\Sigma} \bigl[ \hat{\mathbf{A}}_{T}
\bigl( \mathbf {X} ( \omega ) \bigr) =\hat A_{T}(\omega) \bigr] = 1.
\end{equation}
In fact, $\hat{\mathbf{A}}_{T}$ is explicitly given, for $( X,
\mathbb{X} ) \in\FullMeasureSet\subset\rpspace^\alpha_g$, by
\begin{eqnarray*}
\hat{\mathbf{A}}( X, \mathbb{X} ) := \mathbf{I}_{T}^{-1}( X
) \mathbf{S}_{T}( X, \mathbb{X} ),
\end{eqnarray*}
where
%
\begin{eqnarray*}
\mathbf{I}_{T}( X ) &:=& \int_{0}^{T}h
( X_{s} )^T C^{-1}(X_{s})
h(X_{s}) \,ds,
\\
\mathbf{S}_{T}( X, \mathbb{X} )_{i,j} &:=& \int
_{0}^{T} h_i(X_{s})^T
C^{-1}_{j \cdot}(X_{s}) \,d\mathbf{X}_{s}
\\
&&{}-
\frac{1}{2} \int_0^T \operatorname{Tr}
\bigl[ D\bigl( h_i C^{-1}_{j \cdot
}\bigr)
(X_s) \Sigma(X_s) \Sigma(X_s)^T
\bigr] \,ds
\end{eqnarray*}
and the $d\mathbf{X}$ integral is understood as a (deterministic)
rough integration against $\mathbf{X} = ( X, \mathbb{X} )$.

\item[(iii)]
The map $\hat{\mathbf{A}}_{T}$ is also continuous with respect to the
volatility specification. Indeed, fix $c>0$ and set
\[
\Xi:= \bigl\{ \Sigma\in C^2_b : c^{-1}I\leq
\Sigma\Sigma^{T}\leq cI \bigr\} .
\]
Then $\hat{\mathbf{A}}_{T}$ viewed as map from $\FullMeasureSet
\times\Xi
\rightarrow\mathbb{R}^{d}$ is continuous.
\end{longlist}
\end{theorem}

\begin{example}
\label{rem:OU}
The case of the $d$-dimensional Ornstein--Uhlenbeck process
\begin{eqnarray*}
dX_t = A f(X_t) \,dt + \Sigma(X_t)
\,dW_t,
\end{eqnarray*}
with $A \in L( \R^d, \R^d)$, $f: \R^d \to\R^d$ is covered by our
setting by taking
$\mathbb{V} = L(\R^d,\R^d)$ and $h = I \otimes f$, in coordinates
\[
\bigl(h_{i}^{k,j}\bigr)= \bigl( f^{j}
\delta_{i}^{k} \bigr),
\]
so that (with summation over up-down indices)
\[
h_{i}^{k,j} ( x ) A_{j}^{i}=A_{j}^{k}
f^{j} ( x ).
\]

In this case, the nondegeneracy condition in point (i) is, for example,
satisfied if
the set of critical points of $f$ has no accumulation points
[i.e., on every bounded set, there is only a finite set of points at which
$\det Df ( x  ) =0$], which can be seen by an application
of the (functional) law of the iterated logarithm for diffusions
(Strassen's law), for example, Proposition 4.1 in \cite{bibCaramellino1998}.
\end{example}

\begin{remark}
Note that Proposition \ref{prop:stratonovichRobust}
can be regarded as a corollary of Theorem \ref{thm:mainTheorem} and
Proposition \ref{prop:lift}.
\end{remark}

\begin{remark}
The continuity statements in (ii) and (iii) also hold with respect to
$p$-variation metric, $p \in(2,3)$. This and other rough path metrics
are discussed in Section~\ref{rpreview}.
\end{remark}

\begin{remark}
By \eqref{eq:identity} the well-known asymptotic properties of the
maximum likelihood estimator $\hat A_T$ like
consistency and asymptotic normality (see, e.g., \cite{Kutoyants2004})
also hold for $\hat{\mathbf{A}}_{T}$.
\end{remark}

\begin{remark}
We briefly discuss in what sense Theorem \ref{thm:mainTheorem}
provides answers to (Q1)--(Q3) above:
\begin{longlist}[(Q2)]
\item[(Q1)] Proposition \ref{prop:wellposednessMLE} gives a pathwise
condition for existence of the MLE in terms of the drift coefficient $h$.

\item[(Q2)] The discussion in Section~\ref{sec:counterexamples} shows
that the classical MLE violates the pathwise stability property that
(Q2) asks for.
Theorem \ref{thm:mainTheorem} shows that by considering the signal $X$
as a rough path we can construct a continuous estimator $\hat{\mathbf
{A}}_{T}$ that overcomes this difficulty.

\item[(Q3)] The question of stability in the volatility coefficient
$\sigma$ can also be solved by moving to a rough path space. Indeed,
Theorem \ref{thm:mainTheorem}(iii) shows that $\hat{\mathbf
{A}}_{T}^\sigma$ is continuous with respect to the observations and
the volatility coefficient. Here, the pathwise approach is crucial,
since in the classical setting it is not even clear how to define the
estimator as a mapping in both variables whereas in the rough paths
approach this is an obvious consequence.
\end{longlist}
\end{remark}

\begin{remark}
While our answer to (Q2) above is best possible, in the sense that one
cannot hope for pathwise stability without
introducing rough paths (see the explicit counterexample in
Section~\ref{sec:counterexamples}),
it leaves the user with the
question of how to exactly understand discrete or continuous data as a
rough path.

In essence, this amounts to measuring the L\'evy area associated to an
observed path. In this direction, there are in fact cases where
the measurement of the area is feasible within the physical system
under observation;
see \cite{bibbailleulDiehl}.
\end{remark}

\section{Proof of the main result}\label{sec:proofs}

To recall, let $W$ be $d$-dimensional Wiener process on $(\Omega
,\mathcal{F},(\mathcal{F}_{t})_{t\geq0},\mathbb{P}), $ $%
A\in\mathbb{V}$ (some fixed finite-dimensional vector space) and
\[
h:\mathbb{R}^{d}\rightarrow L \bigl( \mathbb{V},\mathbb{R}^{d}
\bigr) ,\qquad 
 \Sigma:\mathbb{R}%
^{d}\rightarrow L
\bigl( \mathbb{R}^{d},\mathbb{R}^{d} \bigr)
\]
are Lipschitz continuous coefficients, so that the stochastic
differential equation
%
\begin{eqnarray}\label{diffeq}
dX_{t}& =&h(X_t)A\,dt
+\Sigma ( X_t )
\,dW_{t},\qquad t\in\mathbb{R}_{+},
\nonumber\\[-8pt]\\[-8pt]\nonumber
X_{0}& =&x_{0},
\nonumber
\end{eqnarray}
has a unique solution. We are interested in estimation of $A$, as
function of
some observed sample path $X=X ( \omega ) : [
0,T ]
\rightarrow\mathbb{R}^{d}$ when the coefficients $f$ and $\Sigma$
are known. 
%
\begin{lemma}
\label{lem:girsanov}
Write $\mathbb{P}=\mathbb{P}^{A,\Sigma}$ for the path-space
measure induced by the solution $X$ to \eqref{diffeq}. Assume $%
C=\Sigma\Sigma^{T}$ is nondegenerate (say $%
c^{-1}I\leq C^{-1}\leq cI$ for some $c>0$). Then the $\mathbb{V}$%
-valued MLE (relative to $\mathbb{P}^{0}$), $A=\hat{A}_{T}$, is
characterized by
%
\begin{equation}
I_{T}\hat{A}_{T}=S_{T},
\end{equation}
where
\[
S_{T}=\int_{0}^{T}h(X_s)^{\top}C^{-1}(X_{s})
\,dX_{s} \in\mathbb{V} 
^{\ast}
\]
and
\[
I_{T}=\int_{0}^{T}h(X_s)^{\top}C^{-1}(X_{s})h(X_s)
\,ds\in L \bigl( \mathbb{V},\mathbb{V}^{\ast} \bigr) .
\]
\end{lemma}

\begin{pf}
The statement follows from Girsanov's theorem;
see, for example, \cite{bibJacodShiryaev}, Theorem III.5.34.
\end{pf}

\begin{pf*}{Proof of Proposition \protect\ref{prop:wellposednessMLE}}
The first statement follows from Lemma \ref{lem:girsanov}.
Now we need to understand when $I_{T}$ is nondegenerate. To this end,
pick any
nonzero \mbox{$M \in\mathbb{V}$}. Then, with $g=hM$
we have
\[
\langle M,I_{T}M \rangle=\int_{0}^{T}
\bigl\langle g,C^{-1}g \bigr\rangle \,ds\geq0
\]
and since $ \langle g,C^{-1}g \rangle\geq0$ we see that $%
 \langle M,I_{T}M \rangle$ vanishes iff
\[
\bigl\langle g,C^{-1}g \bigr\rangle= \bigl\langle h ( X_{\cdot
}
)M ,C^{-1} ( X_{\cdot} ) h ( X_{\cdot} ) M \bigr\rangle
\equiv0
\]
on $ [ 0,T ]$. Thanks to (assumed) nondegeneracy of $C$
this happens iff
\[
h ( X_{\cdot} ) M \equiv0
\]
on $ [ 0,T ]$.
Hence, for every path in $R_h$, $I_T$ is invertible.
\end{pf*}

\begin{pf*}{Proof of Theorem \protect\ref{thm:mainTheorem}}
(i)~Follows as combination of Proposition \ref
{prop:wellposednessMLE} and Proposition \ref{prop:lift}.

(ii) Recall that for $( X, \mathbb{X} ) \in\rpspace
^\alpha_g$ we have
\begin{eqnarray*}
\hat{\mathbf{A}}( X, \mathbb{X} ) := \mathbf{I}_{T}^{-1}( X
) \mathbf{S}_{T}( X, \mathbb{X} ),
\end{eqnarray*}
where
\begin{eqnarray*}
\hspace*{-3pt}&& \mathbf{I}_{T}( X ) := \int_{0}^{T}h
( X_{s} )^T C^{-1}(X_{s}) h(X_{s}) \,ds,
\\
\hspace*{-3pt}&& \mathbf{S}_{T}( X, \mathbb{X} )_{i,j} := \sum
_k \int_{0}^{T} h_i(X_{s})^T C^{-1}_{jk}(X_{s})
\,d\mathbf {X}^k_{s},
\\
\hspace*{-3pt}&&{-}\sum_k \frac{1}{2} \int
_0^T \sum_{n,m} \bigl[ h_i(X_s) \partial _{x_n}
C^{-1}_{j k}(X_s) + \partial_{x_n}h_i(X_s) C^{-1}_{j k}(X_s)
\bigr] \Sigma_{n,m}(X_s) \Sigma_{k,m}(X_s) \,ds
\\
\hspace*{-3pt}&&\qquad = \sum_k \int_{0}^{T}
h_i(X_{s})^T C^{-1}_{j \cdot}(X_{s})
\,d\mathbf{X}_{s}
\\
\hspace*{-3pt}&&\quad\qquad{}- \frac{1}{2} \int_0^T
\operatorname{Tr}\bigl[ D\bigl( h_i C^{-1}_{j \cdot
}\bigr) (X_s) \Sigma(X_s) \Sigma(X_s)^T
\bigr] \,ds,
\end{eqnarray*}
where the $dX$ integral is understood as a rough path integral
(Section~\ref{rpreview}).
Note that in the definition of $\mathbf{S}_T$ we have formally used
the Stratonovich form $S^{\mathrm{Strat}}_T$,
which is sensible since rough path lift given in Proposition \ref
{prop:lift} is the Stratonovich lift of $X$.
%

Now $\mathbf{S}_T(X,\mathbb{X})$ is continuous in rough path metric
by the just mentioned references.
Moreover, $\mathbf{I}_T(X)$ is obviously continuous in supremum metric,
and hence is its inverse [everywhere defined on $\FullMeasureSet$ by~(i)].

Finally, by Proposition 17.1 in \cite{FV},
$\mathbf{S}_T(X, \mathbb{X})|_{\mathbf{X} = \mathbf{X}(\omega)}$
coincides with $S_T(\omega)$.
$\mathbf{I}_T(X)|_{\mathbf{X} = \mathbf{X}(\omega)}$ trivially
coincides with $I_T(\omega)$, since it only depends on the path (the
first level of the rough path).
Hence, $\hat{\mathbf{A}}_{T} ( \mathbf{X} ( \omega
 )  ) = A_{T}(\omega)$ a.s. under $\mathbb{P}^{0,\Sigma}$.

(iii) This boils down to continuity of the rough integrals
as functions of
integrand $1$-form; see, for example, Theorem 10.47 in \cite{FV}.
\end{pf*}

\section{Misspecification of the noise}\label{sec:misspecification}

In this section, we investigate the behavior of the MLE under
misspecification of the noise $W$ in the sense that we suppose that the
true model has a driving process with nontrivial dependence structure
in time. For the sake of argument, we shall consider \eqref{sde} with
fractional Brownian noise. Fractional noise was first used in
stochastic modeling by Mandelbrot and van Nees in their seminal paper
\cite{Mandelbrot} and is now heavily used in such diverse fields as
the study of turbulence or mathematical finance, see, for example,
\cite{Shao,Cetal}.

For further simplicity, assume $\Sigma\equiv I$ so that the dynamics are
%
\begin{equation}
\label{Hsde} dX_t^H= h \bigl( X_t^H
\bigr) A \,dt+ dW_{t}^H,
\end{equation}
started from a fixed starting point $x_0$, with $W^H$ a
multi-dimensional Volterra fractional Brownian motion with Hurst index
$H \in(0,1)$, that is,
%
\begin{eqnarray}
\label{eq:volterra} W_t^H = \int_0^t
K^H(t,s) \,dW_s,
\end{eqnarray}
where $W$ is a standard Brownian motion, $K^H(t,s) = (t-s)^{H-1/2}$ is
the Volterra kernel.\footnote{The results of this section also hold
true for
classical fractional Brownian motion, using the kernel given in \cite
{bibDecreusefondUstunel1999}. The only difference is that the estimates
in the proof of Theorem \ref{thm:hurstStability} become more
technical.} Note that $W^H|_{H=1/2}=W$ is a standard Brownian motion
and that \mbox{$X^H \to X$}, for example, in probability uniformly on $[0,T]$ as $H\to
1/2$, where
%
\begin{equation}
\label{noHsde} dX_t=h ( X_t ) A \,dt+ dW_{t}.
\end{equation}

[Thanks to additivity of the noise in (\ref{Hsde}) this is a truly
elementary statement, namely a consequence of the continuity of the It\^
o-map as detailed below.]
Suppose now that the \textit{true} dynamics correspond to (\ref{Hsde})
with $H=1/2-\varepsilon$. Clearly, for very small $\varepsilon>0$, the
model (\ref{noHsde}), mathematically much easier, is still an
excellent description of the true dynamics. Indeed, it is well known
that in the additive noise case \eqref{Hsde} or \eqref{noHsde} the
solution map
\[
\bigl(W_t : t \in[0,T]\bigr) \mapsto\bigl(X_t^H
: t \in[0,T]\bigr)
\]
is locally Lipschitz continuous with respect to $\sup$-norm (see,
e.g., \cite{Dembo92},\break page~188).
We can then try to perform classical MLE estimation using the wrong
model (\ref{noHsde}) and write down the estimator $\hat{A}_T=I_T^{-1}S_T$
as was done in (\ref{MLEclassical}).

If we use the It\^o form of the estimator,
the It\^o integrals appearing blow up when applied to fractional
Brownian sample paths ``rougher'' than Brownian motion.%
\footnote{This is well known and in fact easy to see: just consider
the left-point Riemann--Stieltjes approximations to the It\^o-integral
$\int_0^1 W^H \,dW^H$ where $W^H$ is a scalar fractional Brownian
motion. When $H>1/2$ one has convergence to the Young integral
[actually equal to $(1/2)(W^H_1)^2$]. When $H=1/2$ one has convergence
to the It\^o integral. When $H<1/2$ the approximations diverge, as may
be seen by computing their (exploding) variance.}
As pointed out in Section~\ref{sec:stratonovich},
the Stratonovich version of the estimator is much more stable.
Using rough path theory, and in particular our rough path estimator
$\hat{\mathbf A}_T$, we can show not only that
the estimator remains well defined when $H=1/2-\varepsilon$, but also
that it behaves continuously in $H$. This is
spelled out fully in the following theorem.

\begin{theorem}
\label{thm:hurstStability}
Suppose that $H \in(1/3,1)$. Then, for every $\alpha\in(1/3,H)$,
there exists a geometric $\alpha$-H\"older rough path lift $\mathbf
{X^H} = (X^H, \mathbb{X^H})$ of $X^H$ (natural in the sense that
$\mathbf{X}^H$ is the common rough path limit, in probability, of
piecewise linear, mollifier or Karhunen--Loeve approximations to
$X^H)$. Moreover, there is a continuous modification of $\mathbf
{X}^H:H \in(1/3,1)$. As a consequence, $\hat{\mathbf A}_T (\mathbf
{X}^H)$ is well defined and robust with respect to the Hurst parameter,
\[
\hat{\mathbf A}_T \bigl(X^H, \mathbb{X}^H
\bigr) \to\hat{\mathbf A}_T (X, \mathbb{X})
\]
almost surely, as $H \to1/2$, where $(X, \mathbb{X})$ is the lift
$\mathbf{X}$ of $X$ from Theorem~\ref{thm:mainTheorem}.
\end{theorem}

\begin{pf}
Without loss of generality $T=1$. It is a well-known fact (Section~15
in \cite{FV}) that for fixed $H \in(1/3,1]$, $X^H$ can be lifted to
an $\alpha$-H\"older rough path $\mathbf{X}^H = (X^H, \mathbb{X}^H)$.

We will apply Kolmogorov's continuity theorem to construct $\mathbf
{W}^H$ that is almost surely continuous in $H$.
First, using \eqref{eq:volterra}
\begin{eqnarray*}
R_{ W^H - W^{H'} }(s,t) &=& \mathbb{E}\bigl[ \bigl(W^H_s
- W^{H'}_s\bigr) \bigl(W^H_t -
W^{H'}_t\bigr) \bigr]
\\
&\le& \sup_{t\in[0,1]} \mathbb{E}\bigl[ \bigl(W^H_t
- W^{H'}_t\bigr)^2 \bigr]
\\
&=& \sup_{t\in[0,1]} \int_0^t
\bigl( \llvert t-r\rrvert ^{H-1/2} - \llvert t-r\rrvert ^{H'-1/2}
\bigr)^2 \,dr
\\
&=& \int_0^1 \bigl( r^{H-1/2} -
r^{H'-1/2} \bigr)^2 \,dr
\\
&=&
O \bigl(\bigl|H-H'\bigr|^2\bigr).
\end{eqnarray*}
We can now apply Remark 15.38 in \cite{FV} to get
\begin{eqnarray*}
\mathbb{E}\bigl[ \rho_\alpha\bigl( \mathbf{W}^H,
\mathbf{W}^{H'} \bigr)^q \bigr] \le C \bigl\llvert
H-H'\bigr\rrvert ^{\theta},
\end{eqnarray*}
for some $q,C$ large enough and $\theta>0$ small enough.
Applying Kolmogorov's continuity criterion, we get a version of
$\mathbf{W}^H$ that is continuous in $H$.
Since $\mathbf{X}^H$ is the solution to a rough differential equation
driven by $\mathbf{W}^H$, that is, the continuous image of $\mathbf
{W}^H$, it is clear that
$\mathbf{X}^H$ is also continuous in $H$ (with respect to \mbox{$\alpha$-}H\"older rough path topology). The convergence of $\hat{\mathbf A}_T
(X^H, \mathbb{X^H})$ follows now from Theorem \ref{thm:mainTheorem}(ii).
\end{pf}

\section{Failure of continuity for the classical MLE} \label
{sec:counterexamples}

We consider the two-dimensional Ornstein--Uhlenbeck process. This class
of processes was first used by Ornstein and Uhlenbeck to describe the
movement of a particle due to random impulses known as physical
Brownian motion (see \cite{FGL} for a detailed analysis in a rough
path context).
Later these dynamics were applied extensively in finance, for example,
to model commodity prices \cite{schwartz} or interest rates, where it
is called the Vasicek model \cite{vasicek}.

More precisely, let $A \in\mathbb{V} := L(\mathbb{R}^2, \mathbb
{R}^2)$, $h(x)=x$ for all $x$, $g \equiv0$ and $\Sigma= I$ and
consider the model\footnote{$\ldots$which of course fits in the framework
of this paper, as pointed out in Example \ref{rem:OU}.}
\begin{eqnarray*}
dX_t &=& A X_t \,dt + dW_t,
\\
X_0 &=& x_0 \in\mathbb{R}^2.
\end{eqnarray*}
%
By Lemma \ref{lem:girsanov}, the (classical) likelihood estimator
$\hat A_T \in\R^{2\times2}$ is obtained from the relation
%
\begin{eqnarray}
\label{eq:abstract} I_T \hat A_T = S_T.
\end{eqnarray}
%

A straightforward calculation gives
%
\begin{eqnarray}
\label{OUmle}
\hat A^{i,j}_T (\omega) = \frac{1}{U(X)}
\biggl( \int_0^T \bigl(X^{\bar
\jmath}_s
\bigr)^2 \,ds \int_0^T
X^j_s \,dX^i_s - \int
_0^T X^i_s
X^{\bar
\imath}_s \,ds \int_0^T
X^{\bar\jmath}_s \,dX^i_s \biggr),\hspace*{-30pt}
\end{eqnarray}
where $\bar\imath:= 3 - i, \bar\jmath:= 3 - j$,
$U(X) = \int_0^T (X^1_r)^2 \,dr \int_0^t (X^2_r)^2 \,dr -  (\int_0^T X^1_r X^2_r \,dr )^2$
and all stochastic integrals are understood in It\^o sense. Note that
this allows us
to see the precise dependence of the MLE on the iterated integrals of
the observation.
The Stratonovich version reads as
\begin{eqnarray*}
\hat A^{\mathrm{Strat},i,j}_T (\omega) &=& \frac{1}{U(X)} \biggl( \int
_0^T \bigl(X^{\bar\jmath}_s
\bigr)^2 \,ds \biggl( \int_0^T
X^j_s \circ dX^i_s - \delta
_{i,j} \frac{T}{2} \biggr)
\\
&&{}  - \int_0^T X^i_s
X^{\bar\imath}_s \,ds \biggl( \int_0^T
X^{\bar\jmath}_s \circ \,dX^i_s -
\delta_{\jmath,i} \frac{T}{2} \biggr) \biggr).
\end{eqnarray*}
As shown in Section~\ref{sec:stratonovich} this estimator,
\emph{defined} on smooth path by replacing Stratonovich with
Riemann--Stieltjes integrals,
possesses a certain continuity in probability.

We now show that \emph{pathwise} stability fails for this
MLE.\footnote
{
In fact, more is true: by the result in \cite{L90}, there exists no
continuous functional $F$ on pathspace,
such that $\hat A^{\mathrm{Strat},i,j}_T (\omega) = F(X)$.
}
To this end, it suffices to consider the case $i=j=1$ and we construct
a sequence of observations paths $X_n$ that converges uniformly to some
limit $X$, but
\[
\bigl\llvert \hat A^{\mathrm{Strat},1,1}_T (X_n) - \hat
A^{\mathrm{Strat},1,1}_T (X)\bigr\rrvert \to\infty,
\]
as $n \to\infty$. This means that observations can be arbitrarily
close in uniform norm, but the corresponding estimates for $A$ diverge.
At the core of this robustness problem lies, as we will see below, the
fact that multi-dimensional iterated integrals (as the ones appearing
in $\hat A_T$) are discontinuous in $\sup$-norm.

We modify the usually given example of ``spinning fast enough around
the origin'' (see, e.g., Section~1.5.2 in \cite{LCT}),
since we want the limiting path to yield an $I_T(X)$ that is invertible.

We start with the path $X: [0,1] \to\R^2$ that goes, at constant
speed, clockwise, through the square with corners $(0,0)$ and $(1,1)$.
This path lies in the set $R_h$ of Theorem \ref{thm:mainTheorem}
(see Remark \ref{rem:OU} for the definition of $h$).

Now we attach a fast spinning loop at the end as follows:
\begin{eqnarray*}
X_n(t) &:=& X \biggl( \frac{n}{n-1} t \biggr), \qquad t \in
\bigl[0, (n-1)/n\bigr],
\\
X_n(t) &:=& \frac{1}{n} \bigl( e^{i 2 \pi n^3 (t - (n-1)/n)} - 1 \bigr),
\qquad t \in\bigl[(n-1)/n,1\bigr].
\end{eqnarray*}

Evaluating the upper left component of the likelihood estimator $\hat
A_T$ from \eqref{OUmle} for $X= (X^{(1)},X^{(2)})$ and $T=1$ yields
\begin{eqnarray*}
\hat
A^{\mathrm{Strat},1,1}_T (X) &=& \frac{1}{U(X)} \biggl( \int
_0^1 X^{(2)}_r
X^{(2)}_r \,dr \biggl( \int_0^1
X^{(1)}_r\, d X^{(1)}_r-
\frac{T}{2} \biggr)
\\
&&{} - \int_0^1
X^{(1)}_r X^{(2)}_r \,dr \biggl( \int
_0^1 X^{(2)}_r\, d
X^{(1)}_r- \frac{T}{2} \biggr) \biggr).
\end{eqnarray*}
The prefactor $U(X)$, consisting only of Riemann integrals, is
continuous in supremum and so $U(X^{(n)})$ converges to a finite limit
as $n \to\infty$.
The same holds true for the first factor in the large bracket (the
stochastic integral is seen to be continuous by an application of It\^
o's formula)
and the factor $\int_0^1 X^{(1)}_r X^{(2)}_r \,dr$ in the last term. Now
for $\int_0^1 X^{(2)}_r\, d X^{(1)}_r$, note first that
\begin{eqnarray*}
\int_0^1 X^{(2)}(r) \,d
X^{(1)}(r) &=& \int_0^{(n-1)/n}
X^{(2)}(r) \,d X^{(1)}(r)
\\
&&{}+ \int_{(n-1)/n}^1
X^{(2)}(r) - X^{(2)}\bigl( (n-1)/n \bigr) \,d
X^{(1)}(r).
\end{eqnarray*}

Moreover, since the $X^{(n)}$ have the same value at $t=0, (n-1)/n,1$
it is easy to see that
\begin{eqnarray*}
\int_0^{(n-1)/n} X^{(2)}(r) \,d
X^{(1)}(r) = \mathcal{A}_{0,(n-1)/n}\bigl( X^{(n)} \bigr),
\\
\int_{(n-1)/n}^1 X^{(2)}(r) \,d
X^{(1)}(r) = \mathcal{A}_{(n-1)/n,1}\bigl( X^{(n)} \bigr),
\end{eqnarray*}
%
where $\mathcal{A}_{s,t}(X)$ is (two times) the area between that curve
$ \{ X (r) \equiv(X_1 (r) ,\break X_2 (r) ) : s \le r \le t \}$ and the chord from
$X (t) $ to $X (s)$.%
%
%
%
\footnote{It is given as
\[
\mathcal{A}_{s,t}(X) = \int_{s}^{t}
\bigl(X_1(r) - X_1(s) \bigr) \,dx_2(r) - \int
_s^t \bigl( X_2(r) -
X_2(s) \bigr) \,dx_1(r).
\]
}
Hence,\break $\mathcal{A}_{0,(n-1)/n}( X^{(n)} ) \equiv-2$
and
$\mathcal{A}_{(n-1)/n,1}( X^{(n)} ) = -\pi n$
and, therefore, as desired,
\[
\bigl\llvert \hat A^{\mathrm{Strat},1,1}_T(X_n)\bigr\rrvert
\to\infty.
\]

In conclusion, the estimation problem is \textit{not well-posed} if one
measures distance of paths in supremum norm.
Let us note that by working in stronger topologies on pathspace, say
$\alpha$-H\"older with $\alpha< 1/2$ (so that they support Wiener
measure), the situation does not
change; see, for example, \cite{L90}.

On the other hand, as was seen in Section~\ref{sec:main}, the
``rough'' estimator $\hat{\mathbf A}_T$ built in
Theorem \ref{thm:mainTheorem} restores continuity. In the present
example, this reduces to the fact that
the iterated stochastic integrals which appear in $\hat A_T(\omega)$
may be rewritten in terms of
rough integrals against the rough path lift of $X$.


\section{Numerical examples}
\label{sec:numerical}
We illustrate our theoretical results in two numerical examples.
The first example uses a fractional Brownian motion with Hurst
parameter $H$ as driving noise.
The It\^o integral is not even well defined for Hurst parameter $H \neq
\frac{1}{2}$,
but we show that the estimator using the rough path lift of the
fractional Brownian motion performs well in this setting.
We use Stratonovich-type Riemann sums to approximate the rough path
lift (and hence are strictly speaking
only performing the robustification laid out in Section~\ref
{sec:stratonovich}),
and hence this example can foremost be seen as a strong encouragement
to use them over It\^o-type approximations.

In the second example, the driving noise is replaced by a physical
Brownian motion in a magnetic field.
On the level of the path this is known to converge to Brownian motion,
but its lifted rough path does not (cf. \cite{FGL}).
We demonstrate that the classical MLE breaks down in this setting and
how a deterministic correction on the second level
of the rough path leads nonetheless to good estimation results.

\subsection{Fractional Ornstein--Uhlenbeck process}


\subsubsection{One-dimensional}

In this section, we demonstrate in the setting of Example \ref{Ex1}
the instability of the MLE due to the It\^o integrals, if one does not
get rid of the stochastic integral via integration by parts.
Furthermore, we show that by using Stratonovich-type approximations as
suggested in Section~\ref{sec:stratonovich} we obtain a stable estimator.

We simulate samples from a one-dimensional fractional
Ornstein--Uhlenbeck process defined by
%
\begin{equation}
dX_t^H=A X_t^H \,dt+
dW_{t}^H, \qquad t \in[0,T].
\end{equation}
%
We use an exact simulation scheme to draw equidistant samples
\[
X_\Delta^H, X_{2 \Delta}^H, \ldots,
X_{n \Delta}^H \qquad\mbox {for } \Delta> 0
\]
from $X^H$ such that $T = n \Delta$. The discretized maximum
likelihood estimator $\hat A_T$ for A in this model is given by
\begin{eqnarray*}
\hat{A}_{T}^n\bigl(X^H\bigr) &=&
\frac{\sum_{i =1}^{n-1} X^H_{i\Delta}
(X^H_{(i+1)\Delta}-X^H_{i\Delta})}{\sum_{i =1}^{n-1} (X^H_{i\Delta
})^2 \Delta}.
\end{eqnarray*}
From Theorem \ref{thm:mainTheorem}, we obtain the discretized rough MLE
\[
\hat{\mathbf{{A}}}_{T}^n\bigl(\mathbf{X}^H
\bigr) = \frac{\sum_{i=1}^{n-1} X_{i\Delta}^H (X^H_{(i+1)\Delta
}-X^H_{i\Delta})+ \mathbb{X}_{i\Delta}-\mathbb{X}_{(i+1)\Delta}
}{\sum_{i =1}^{n-1} (X^H_{i\Delta})^2 \Delta}.
\]
In Figure~\ref{var_H}, Monte Carlo estimates of variances of $\hat
A_T^n$ and the rough MLE $\hat{ \mathbf{A}}_T^n$ are depicted for
varying Hurst index from each 500 Monte Carlo iterations.
The sample size is $n = 100$ (i.e., the time mesh size of observation
is $1/n$) and the time horizon $T = 1$. We clearly see that the
variance increases when $H$ moves away from $1/2$ and explodes for $H$
going to $0$. On the contrary, the rough MLE remains stable on the
whole range of $H$ values with an almost constant variance.

Note that the rate of convergence for the variance of $\hat A_T^n$ is
proportional to $n(1/2- H)$ so that the effect that can be seen in
Figure~\ref{var_H} becomes more severe with growing sample size. This
connection is depicted in Figure~\ref{var_n} where we see Monte Carlo
estimates of the variance of $\hat A_T$ for increasing sample size. The
number of Monte Carlo iterations for each $n$ is $N = 100$, the time
horizon $T = 1$ and the Hurst index $H= 0.35$.
%
\begin{figure}

\includegraphics{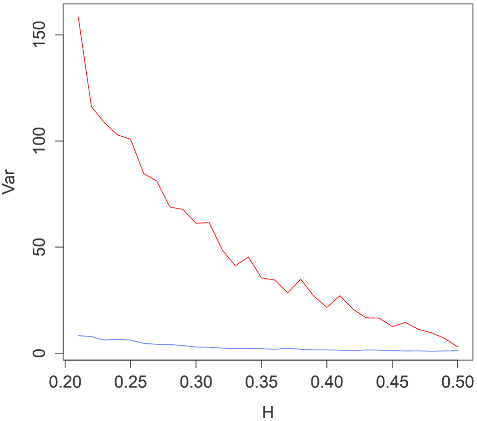}

\caption{Monte Carlo estimate of the variance of the classical MLE
(red) and rough MLE (blue) for different Hurst indices $H$.} \label{var_H}
\end{figure}

\begin{figure}

\includegraphics{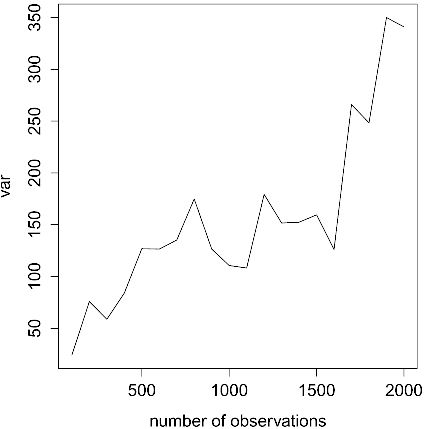}

\caption{Estimated variance of the classical MLE for varying sample
size.} \label{var_n}
\end{figure}

In Table~\ref{tableOU}, the mean and standard deviation of the RMLE
$\hat{\mathbf{{A}}}_{T}^n$ are estimated for the fractional
Ornstein--Uhlenbeck model and for various Hurst indices. Each estimate
consists of 1000 Monte Carlo iterations and the true parameter was
\mbox{$A=2$}. We find that already for quite moderate sample size of $n=100$
the estimator performs very well. We also observe that when $T$ grows a
slight discretization bias appears that is typically of the order
$\Delta$. Surprisingly, the RMLE gives accurate results even when the
Hurst parameter is far away from the classical case at $H = 1/2$.

\begin{table}[b]
\tabcolsep=0pt
\caption{Mean and standard deviation of the RMLE $\hat{\mathbf
{{A}}}_{T}^n$ for varying Hurst indices in the misspecified model}\label{tableOU}
\begin{tabular*}{\tablewidth}{@{\extracolsep{\fill}}@{}lcccccccc@{}}
\hline
&&&\multicolumn{2}{c}{$\bolds{H=0.5}$} & \multicolumn{2}{c}{$\bolds{H=0.4}$} & \multicolumn{2}{c@{}}{$\bolds{H=0.3}$}\\[-6pt]
&&&\multicolumn{2}{c}{\hrulefill} & \multicolumn{2}{c}{\hrulefill} & \multicolumn{2}{c@{}}{\hrulefill}\\
$\bolds{a}$ &$\bolds{T}$ & $\bolds{n}$ & \textbf{Mean} & \textbf{Std dev} & \textbf{Mean} & \textbf{Std dev} & \textbf{Mean} & \textbf{Std dev}\\
\hline
2 & 1 & 100 & 2.0& 0.20\phantom{0} & 2.0& 0.18\phantom{0} &2.0 &0.33\phantom{0} \\
& & 200 & 2.0& 0.17\phantom{0} & 2.0 & 0.18\phantom{0} &2.0 &0.20\phantom{0} \\
& & 500 & 2.0 & 0.16\phantom{0} & 2.0& 0.19\phantom{0} &2.0 &0.19\phantom{0}
\\[3pt]
& 2 & 100 & 2.0 & 0.023 & 2.0 & 0.023 & 2.0 & 0.026 \\
& & 200 & 2.0 & 0.025 & 2.0 & 0.026 & 2.0 & 0.024 \\
& & 500 & 2.0 & 0.022 & 2.0& 0.022 &2.0 & 0.051
\\[3pt]
& 5 & 100 & 2.1 & 8.1e$-$05 & 2.1 & 8.3e$-$05 & 2.1 & 9.2e$-$05 \\
& & 200 & 2.1 & 6.7e$-$05 & 2.1 & 6.6e$-$05 & 2.1 & 7.0e$-$05 \\
& & 500 & 2.0 & 5.6e$-$05 & 2.0 & 5.6e$-$05 & 2.0 & 6.4e$-$05\\
\hline
\end{tabular*}
\end{table}

\subsubsection{Two-dimensional}
Here, we give numerical examples for the two-dimensional
Ornstein--Uhlenbeck dynamics. 
%
We apply a Euler scheme to draw an equidistant sample $X_\Delta^H,
X_{2 \Delta}^H, \ldots, X_{n \Delta}^H$ for $\Delta>0$ from the
process $X$ solving
\begin{eqnarray*}
dX_t &=& A X_t \,dt + dW_t^H,
\\
X_0 &=& x_0 \in\mathbb{R}^2,
\end{eqnarray*}
where $W^H$ is a two-dimensional fractional Brownian motion with Hurst
parameter $H \in(0,1)$ and $A$ is given by
\begin{eqnarray*}
A = \pmatrix{ 1 & 2
\vspace*{3pt}\cr
-2 & 1}.
\end{eqnarray*}

The expression for the classical maximum likelihood estimator [see
\eqref{OUmle}]
is of course only valid for $H=1/2$.
Moreover, for $H < 1/2$ the It\^o integrals appearing in that estimator
are in general not even well defined.
Nonetheless, since we simulate on a discrete time grid, we
can calculate its discretized version, replacing the stochastic
integrals by It\^o-type Riemann sums.

On the other hand, for every $H>1/3$ fractional Brownian motion
possesses a natural rough path lift
(see, e.g., \cite{FV}, Chapter~15),
so the expression for the ``rough'' MLE \eqref{eq:roughMLE}
is at least well defined, also for $H\neq1/2$.
Since we deal with a simulation at discrete timepoints, we have
to approximate this rough path lift.
We shall use Stratonovich-type Riemann sums, which are well known to
converge (see,
e.g.,~\cite{FH}, Chapter~10, and the references therein).
We then plug the result into the rough path estimator \eqref{eq:roughMLE}.

We give the estimation results for the upper right coordinate of $A$
with true value equal to $2$,
on a discrete grid for varying number of observations $n$ and
observation length $T$.

In Table~\ref{table:2DOU}, the estimated mean and standard deviation
for the discretized classical MLE (top) and the discretized ``rough''
MLE (bottom) are given. Each value is based on 100 Monte Carlo runs of
the estimator.

We find that for $H=1/2$ the classical MLE performs well if the
observation length $T$ is large enough. When $H$ moves away from $1/2$
the instability of the estimator becomes apparent. The standard
deviation increases significantly and the estimator is strongly biased.

In contrast to that the ``rough'' MLE $\hat{\mathbf{A}}_{T,n}^{1,2}$
performs equally well over the whole range of Hurst indices. For $H =
1/2$, both estimators give similar results as expected from our results
in Theorem \ref{thm:mainTheorem} whereas in the dependent regime
$\hat{\mathbf{A}}_{T,n}^{1,2}$ clearly outperforms the classical MLE.

\subsection{Physical Brownian motion in a magnetic field}

\begin{table}
\tabcolsep=0pt
\caption{We\vspace*{1pt} consider the ``true'' parameter value $A^{1,2}=2$ and give
estimates of the mean and standard deviation of the classical MLE
${\hat{{A}}}_{T,n}^{1,2}$ (top) and the ``rough'' MLE $\hat{\mathbf
{A}}_{T,n}^{1,2}$ (bottom) based on $100$ Monte Carlo iterations for
varying Hurst indices for the 2-dim. OU process. Here, $n$  is the
number of time-steps in the Euler approximation}\label{table:2DOU}
\begin{tabular*}{\tablewidth}{@{\extracolsep{\fill}}@{}lccd{1.2}cd{1.2}cd{1.2}@{}}
\hline
&&\multicolumn{2}{c}{$\bolds{H=0.5}$} & \multicolumn{2}{c}{$\bolds{H=0.4}$} & \multicolumn{2}{c@{}}{$\bolds{H=0.35}$}\\[-6pt]
&&\multicolumn{2}{c}{\hrulefill} & \multicolumn{2}{c}{\hrulefill} & \multicolumn{2}{c@{}}{\hrulefill}\\
$\bolds{T}$ & $\bolds{n}$ & \textbf{Mean} & \multicolumn{1}{c}{\textbf{Std dev}} & \textbf{Mean} & \multicolumn{1}{c}{\textbf{Std dev}} & \textbf{Mean} & \multicolumn{1}{c@{}}{\textbf{Std dev}}\\
\hline
 \phantom{0}1 & 100 & 2.7 & 1.6 & 3.4 & 2.0 & 5.5 & 2.9 \\
 & 200 & 2.7 & 1.8 & 3.7 & 2.6 & 6.9 & 3.2 \\
 & 500 & 2.7 & 1.8 & 4.8 & 2.7 & 9.7 & 4.6 \\[3pt]
 \phantom{0}5 & 100 & 2.1& 0.90 & 2.7& 1.13 &3.9 &1.1 \\
 & 200 & 2.2& 0.99 & 3.1 & 1.1 &4.9 &1.4 \\
 & 500 & 2.2 & 1.06 & 3.9& 1.3 &6.5 &1.9 \\[3pt]
 10 & 100 & 2.0 & 0.61 & 2.5 & 0.70 & 3.0 & 0.43 \\
 & 200 & 2.0 & 0.73 & 2.7 & 0.75 & 3.9 & 0.54 \\
 & 500 & 2.1 & 0.74 & 3.3& 0.90 &5.1 & 0.90
\\[6pt]
 \phantom{0}1 & 100 & 2.7 & 1.6 & 2.3 & 1.6 & 2.1 & 2.2 \\
 & 200 & 2.8 & 1.6 & 2.4 & 2.0 & 2.2 & 1.9 \\
 & 500 & 2.7 & 1.6 & 2.3 & 1.9 & 2.2 & 2.2 \\[3pt]
 \phantom{0}5 & 100 & 2.2& 1.0 & 2.1& 1.2 &1.8 &1.1 \\
 & 200 & 2.3& 1.0 & 1.9 & 1.1 &1.8 &1.2 \\
 & 500 & 2.2 & 1.0 & 1.9& 1.1 &1.9 &1.3 \\[3pt]
 10 & 100 & 2.0 & 0.64 & 1.8 & 0.71 & 1.8 & 0.83 \\
 & 200 & 2.1 & 0.75 & 1.9 & 0.86 & 1.8 & 0.87 \\
 & 500 & 2.0 & 0.75 & 2.0& 0.91 &1.9 & 0.89
 \\
 \hline
\end{tabular*}
\end{table}

The dynamics of
a two-dimensional physical Brownian motion $\bar{W}^{\alpha,m}$
in a magnetic field are given by (see~\cite{FGL})
\begin{eqnarray*}
d\bar{W}^{\alpha,m}_t &=& \frac{1}{m} P^{\alpha,m}_t
\,dt,
\\
dP^{\alpha,m}_t &=& - \frac{1}{m} M P^{\alpha,m} \,dt +
dW_t.
\end{eqnarray*}
Here,
\begin{eqnarray*}
M = M_\alpha = \pmatrix{ 1 & \alpha
\vspace*{3pt}\cr
-\alpha& 1},
\end{eqnarray*}
with strength of the magnetic field given by the scalar parameter
$\alpha$
and mass $m>0$ of the particle, assumed to carry unit charge.

As in the preceding section, we observe the realization of an
Ornstein--Uhlenbeck process,
but now driven by the physical Brownian motion and with covariance
matrix $M$:
%
\begin{eqnarray}
\label{equ:OULimtPBM} dX^{\alpha,m}_t = A X^{\alpha,m}_t
\,dt + M \,d\bar{W}^{\alpha,m}_t.
\end{eqnarray}

Now, it is quite easy to show (see, e.g.,
\cite{pavliotisStuart}, Section~11.7.7)
that $M \bar{W}^{\alpha,m} \to W$ in supremum norm, as $m\to0$.
In the one-dimensional case (where the MLE is continuous in supremum
norm, as we saw in Example \ref{Ex1}),
it automatically follows that
\begin{eqnarray*}
\hat A_T\bigl( X^{\alpha,m} \bigr) \mathop{\rightarrow}_{m \to0}
\hat A_T( X ),
\end{eqnarray*}
where $X$ solves the classical OU equation $dX = A X \,dt + dW$
and $\hat A_T$ is the MLE for the latter.
On the other hand, in dimension $d \ge2$, and in presence of a
magnetic field $\alpha\neq0$ we still have $M \bar{W}^{\alpha,m}
\to W$ but
the desired convergence \textit{fails}, that is,
\begin{eqnarray*}
\hat A_T\bigl( X^{\alpha,m} \bigr) \mathop{\nrightarrow}_{m \to0}
\hat A_T( X ).
\end{eqnarray*}
The underlying reason is failure of convergence at the level of
iterated integrals to the
Stratonovich iterated integrals of $W$. Instead, as was shown in detail
in \cite{FGL}, one has
\begin{eqnarray*}
\biggl(M \bar{W}^{\alpha,m}_s, \int_s^t
M \bar{W}^{\alpha,m}_{s,r} \otimes dM \bar{W}_r\biggr)
\to \biggl(W_s, \int_s^t
W_{s,r} \otimes\circ W_r + (t-s) D \biggr),
\end{eqnarray*}
with correction term
\begin{eqnarray*}
D := \frac{1}{2} \operatorname{Anti}[M] \operatorname{Sym}[M]^{-1}
= \frac{1}{2} \pmatrix{ 0 & \alpha
\vspace*{3pt}\cr
-\alpha& 0}.
\end{eqnarray*}
As an easy consequence,
\begin{eqnarray*}
\biggl(X^{\alpha,m}_s, \int_s^t
X^{\alpha,m}_{s,r} \otimes d X^{\alpha,m}_r\biggr)
\to \biggl(X_s, \int_s^t
X_{s,r} \otimes\circ X_r + (t-s) D \biggr),
\end{eqnarray*}
and so, by Theorem \ref{thm:mainTheorem}, we have a ``modified'' MLE
stability of the form
\begin{eqnarray*}
\hat{\mathbf{A}}_T\biggl( \biggl(X^{\alpha,m}_s,
\int_s^t X^{\alpha,m}_{s,r}
\otimes d X^{\alpha,m}_r - (t-s) D\biggr) \biggr) \to \hat{
\mathbf{A}}_T\biggl( \biggl(X_s, \int
_s^t X_{s,r} \otimes\circ
X_r\biggr) \biggr),
\end{eqnarray*}
where, with $\bar\imath:= 3 - i, \bar\jmath:= 3 - j$ and
$U(X) = \int_0^T (X^1_r)^2 \,dr \int_0^t (X^2_r)^2 \,dr -\break  (\int_0^T X^1_r X^2_r \,dr )^2$,
%
\begin{eqnarray}
\label{eq:roughMLE}
\hat{\mathbf{A}}^{i,j}_T (X, \mathbb{X}) &=&
\frac{1}{U(X)} \biggl( \int_0^T
\bigl(X^{\bar\jmath}_s\bigr)^2 \,ds \biggl\{ \mathbb
{X}^{i,j}_{0,T} - \delta_{i,j} \frac{T}{2}
\biggr\}
\nonumber\\[-8pt]\\[-8pt]\nonumber
&&{} - \int_0^T X^i_s
X^{\bar\imath}_s \,ds \biggl\{ \mathbb{X}^{\jmath,
i}_{0,T}
- \delta_{\jmath,i} \frac{T}{2} \biggr\} \biggr).
\end{eqnarray}

In summary, it is \textit{perfectly justified}, in the small mass regime
$m \ll1$, to consider the effective dynamics
$dX_t = A X_t \,dt + dW_t$ as approximation for \eqref{equ:OULimtPBM}.
However, it would be {\it wrong} to use the resulting MLE estimator on
the realizations of
realization of $X^{\alpha,m}$, even in the limit $m \to0$. Instead,
the estimation procedure based on $X^{\alpha,m}$
{\it must} take account of the correction term $D$ we exhibited above.
At last, we support our findings with concrete
numerical results, taking
\begin{eqnarray*}
A = \pmatrix{ -3 & 2
\vspace*{3pt}\cr
0 & -4},
\end{eqnarray*}
and $100$ Monte Carlo simulations.
The force of the magnetic field is chosen as $\alpha= 1.0$, the mass
of the particle as $m=0.01$
and discretization is done on a time grid of $10^5$ equidistant points.

\begin{table}[t]
\tabcolsep=0pt
\caption{Mean and standard deviation of $\hat A_T^{1,2}$ over $100$
Monte Carlo runs for the physical Brownian motion model. The correct
value is $2.0$}\label{tab:physical}
\begin{tabular*}{\tablewidth}{@{\extracolsep{\fill}}@{}lcccccccc@{}}
\hline
& \multicolumn{2}{c}{$\bolds{T=1.0}$} & \multicolumn{2}{c}{$\bolds{T=3.0}$} & \multicolumn{2}{c}{$\bolds{T=10.0}$}& \multicolumn{2}{c@{}}{$\bolds{T=30.0}$}\\[-6pt]
& \multicolumn{2}{c}{\hrulefill} & \multicolumn{2}{c}{\hrulefill} & \multicolumn{2}{c}{\hrulefill}& \multicolumn{2}{c@{}}{\hrulefill}\\
& \textbf{Mean} & \textbf{Std dev} & \textbf{Mean} & \textbf{Std dev} & \textbf{Mean} & \textbf{Std dev} & \textbf{Mean} & \textbf{Std dev}\\
\hline
w/correction & 2.0 & 0.84 & 2.0 & 0.72 & 2.1 & 0.59 & \phantom{$-$}2.2 & 0.45 \\
w/o correction & 1.7 & 0.79 & 1.1 & 0.64 & 0.3 & 0.49 & $-$0.4 & 0.35\\
\hline
\end{tabular*}
\end{table}

The results are shown in Table~\ref{tab:physical}.
As is clearly visible, the corrected estimator yields good results,
with decreasing standard variation for increasing time horizons.
The uncorrected estimator on the other hand yields useless results for
times larger then $3.0$.




%

\printaddresses
\end{document}